\documentclass{amsart}
\usepackage{amssymb, mathtools}
\usepackage{esint}
\usepackage{enumerate}
\usepackage[breaklinks,colorlinks]{hyperref}
\usepackage{nameref}
\usepackage{cleveref}

\newtheorem{theorem}{Theorem}[section]
\newtheorem{lem}[theorem]{Lemma}
\newtheorem{cor}[theorem]{Corollary}
\newtheorem{prop}[theorem]{Proposition}

\theoremstyle{definition}

\theoremstyle{remark}

\numberwithin{equation}{section}

\crefname{defn}{Definition}{Definitions}
\crefname{propy}{Property}{Properties}
\crefname{eg}{Example}{Examples}
\crefname{prob}{Problem}{Problems}
\crefname{conj}{Conjecture}{Conjectures}
\crefname{known}{Theorem}{Theorems}

\newcommand{\norm}[1]{\left\lVert#1\right\rVert}
\newcommand{\abs}[1]{\left\lvert#1\right\rvert}
\newcommand{\set}[1]{\left\{#1\right\}}
\newcommand{\hin}[2]{\left\langle#1,#2\right\rangle}
\newcommand*{\To}{\to }

\newcommand*{\Rmn}[1]{\uppercase\expandafter{\romannumeral#1}}
\newcommand*{\dif}{\mathop{} \mathrm{d}}

\DeclareMathOperator{\Div}{div}

\allowdisplaybreaks

\title{Liouville theorem for quasi-harmonic sphere}
\author{Jiayu Li}
\address{Jiayu Li, School of Mathematics, Nanjing University, 22 Hankou Road, Nanjing 210093, China}
\email{jiayuli@nju.edu.cn}

\author{Linlin Sun}
\address{School of Mathematics and Computational Science \& Hunan Research Center of the Basic Discipline Fundamental Algorithmic Theory and Novel Computational Methods, Xiangtan University, Xiangtan 411105, China}
\email{sunll@xtu.edu.cn}


\date{\today}

\begin{document}

\begin{abstract}
Let $N$ be a compact Riemannian manifold whose universal covering supports a convex function. We prove that it admits no quasi-harmonic sphere. No global growth condition is imposed on the convex function.
\end{abstract}

\maketitle

\section{Introduction}

Following Lin and Wang \cite{Lin:1999ab}, a nonconstant smooth map $v:\mathbb R^m\to N$, $m>2$, is called a \emph{quasi-harmonic sphere} if
\begin{equation}\label{eq:shrinker-intro}
 \tau(v)=\dfrac 12 x\cdot \dif v
\end{equation}
and
\begin{equation}\label{eq:finite-gaussian-intro}
 \int_{\mathbb R^m}\abs{\dif v}^2e^{-\abs{x}^2/4}<\infty.
\end{equation}
Equivalently, $v$ is a finite-energy harmonic map from $\left(\mathbb R^m,e^{-\frac {\abs{x}^2}{2(m-2)}}g_0\right)$ to $N$. 

Quasi-harmonic spheres enter the singularity analysis of the harmonic map heat flow \cite{Lin:1999ab}. See also \cite{Ding:2006aa} for related results on equivariant quasi-harmonic spheres. The absence of harmonic spheres and quasi-harmonic spheres plays an important role in the existence and convergence theory for the harmonic map heat flow \cite{Li:2009aa,Lin:1999ab}. Li and Wang proved constancy when the image is contained in a regular ball \cite[Theorem~3.2]{Li:2009aa}.
Li and Zhu excluded quasi-harmonic spheres under a polynomial-growth
strictly convex function on the universal covering \cite{Li:2010aa}, and Li
and Yang treated the corresponding large-energy condition \cite{Li:2012aa}.
In our earlier paper \cite{Li:2016aa}, the large-energy condition was shown to
be equivalent to \eqref{eq:finite-gaussian-intro}, while polynomial growth
was replaced by
\begin{equation}\label{eq:LiSun-growth}
 \widetilde\nabla^2\rho>0,\qquad
 0\le \rho(y)\le C\exp\left(\dfrac 14\widetilde d(y,y_0)^{2/m}\right).
\end{equation}
The question considered here is whether the growth condition in
\eqref{eq:LiSun-growth} can be removed when the image is relatively
compact.

Let $\pi:\widetilde N\to N$ be a Riemannian covering. Following
\cite[Section~2.1]{Jost:2012aa}, we call $\widetilde N$ \emph{convex supporting} if every compact subset of $\widetilde N$ has a neighborhood carrying a $C^2$ function $\tilde\rho$ with $\tilde\nabla^2\tilde\rho>0$.  

Our main result can be stated as follows.

\begin{theorem}\label{thm:shrinker}
If $N$ admits a Riemannian covering $\pi:\widetilde N\to N$ whose total space is convex supporting, then $N$ admits no quasi-harmonic sphere with relatively compact image.
\end{theorem}

Compared with \cite{Li:2016aa}, the globally defined strictly convex function satisfying \eqref{eq:LiSun-growth} is replaced by the local convex-supporting condition, at the price of assuming that the image is
relatively compact. This range assumption is automatic when $N$ is closed. Neither $N$ nor the covering space is assumed complete. In particular, no globally defined strictly convex function or growth condition on the local convex functions is required.

\vspace{2ex}

Every quasi-harmonic sphere gives rise to the backward self-similar flow
\begin{equation}\label{eq:selfsimilar-intro}
 u(x,t)=v\left(\dfrac {x}{\sqrt{-t}}\right),\qquad t<0,
\end{equation}
which is an ancient solution of the harmonic map heat flow introduced by
Eells and Sampson \cite{Eells:1964aa},
\begin{align*}
     \partial_tu=\tau(u), \qquad u:\mathbb R^m\times(-\infty,0)\To N.
\end{align*}
Throughout this paper, an \emph{ancient solution} means a smooth solution defined on $\mathbb R^m\times(-\infty,0)$. When weak solutions are considered, this will be stated explicitly. Ancient solutions of the harmonic map heat flow and related variants have also been studied in \cite{Chen:2023aa,Chen:2026aa,Sung:2022aa,Wang:2011aa} and the references therein.

Put
\begin{align*}
    G_r(x)=(4\pi r^2)^{-m/2}e^{-\abs{x}^2/(4r^2)},\qquad r>0.
\end{align*}
The preceding theorem follows from the following Liouville theorem for ancient harmonic map heat flows.

\begin{theorem}\label{thm:ancient}
Let $m\ge2$ and let $u:\mathbb{R}^m\times(-\infty,0)\to N$ be an ancient harmonic map heat flow with $u(\mathbb{R}^m\times(-\infty,0))\subseteq K\subseteq N$, where $K$ is compact. Suppose that $N$ admits a Riemannian covering
$\pi:\widetilde N\to N$ whose total space is convex supporting. If
\begin{align}\label{eq:energybound}
   \sup_{t<0}(-t)\int_{\mathbb{R}^m}\abs{\dif u(x,t)}^2G_{\sqrt{-t}}(x)\dif x<\infty,
\end{align}
then $u$ is constant.
\end{theorem}

By \eqref{eq:monotonicity}, condition \eqref{eq:energybound} is equivalent to
\begin{align*}
\lim_{t\to-\infty}(-t)\int_{\mathbb R^m}\abs{\dif u(x,t)}^2G_{\sqrt{-t}}(x)\,\dif x<\infty.
\end{align*}
Moreover, Proposition~\ref{prop:global} shows that, for every $a\in\mathbb R^m$ and $T\in\mathbb R$,
\begin{align*}
 &\lim_{t\to-\infty}(T-t)\int_{\mathbb R^m}\abs{\dif u(x,t)}^2G_{\sqrt{T-t}}(x-a)\,\dif x\\
 =&\lim_{t\to-\infty}(-t)\int_{\mathbb R^m}\abs{\dif u(x,t)}^2G_{\sqrt{-t}}(x)\,\dif x .
\end{align*}
Thus the choice of the space-time center in Theorem~\ref{thm:ancient} is only a normalization.

    The natural-scale bound \eqref{eq:energybound} in Theorem~\ref{thm:ancient} is essential.  For
$N=\mathbb S^1$, the universal cover $\mathbb R$ is convex supporting, but
$u(x,t)=e^{ix_1}$ is a nonconstant stationary flow and
\begin{align*}
(-t)\int_{\mathbb R^m}\abs{\dif u(x,t)}^2G_{\sqrt{-t}}(x)\,\dif x=-t\to \infty \qquad(t\to-\infty).
\end{align*}

\section{Local estimates}

Let $u:\mathbb{R}^m\times(-\infty,0)\to N$ be an ancient  harmonic map heat flow and set
\begin{align*}
    \Phi(t)=\dfrac{-t}{2}\int_{\mathbb{R}^m}\abs{\dif u(x,t)}^2G_{\sqrt{-t}}(x)\dif x.
\end{align*}
We begin with Struwe's monotonicity formula \cite[Lemma~3.2, pp.~489--490]{Struwe:1988aa}.  Since only local boundedness of $\Phi$ is assumed, the cutoff argument is included.

\begin{prop}\label{prop:monotonicity}
If $\Phi$ is locally bounded on $(-\infty,0)$, then it is locally
absolutely continuous and, for $t_1<t_2<0$,
\begin{equation}\label{eq:monotonicity}
 \Phi(t_2)+\int_{t_1}^{t_2}(-t)\int_{\mathbb{R}^m}
 \abs{\partial_tu+\dfrac {1}{2t}x\cdot\dif u}^2G_{\sqrt{-t}}(x)\dif x\dif t
 =\Phi(t_1).
\end{equation}
\end{prop}

\begin{proof}
Fix $t_1<t_2<0$ and write $G=G_{\sqrt{-t}}$.
The local energy and stress identities are
\begin{align*}
\dfrac 12\partial_t\abs{\dif u}^2
 =\sum_i\partial_i\hin{\partial_tu}{u_i}-\abs{\partial_tu}^2, \quad \sum_i\partial_i\left(\hin{u_i}{u_k}-\dfrac 12\abs{\dif u}^2\delta_{ik}\right)=\hin{\partial_tu}{u_k}.
\end{align*}
Since $\partial_iG=x_iG/(2t)$ and
$\partial_t((-t)G)=((m-2)/2+\abs{x}^2/(4t))G$, these identities above give
\begin{align*}
    &\partial_t\left(\dfrac {-t}{2}\abs{\dif u}^2G\right)
 +(-t)G\abs{\partial_tu+\dfrac {1}{2t}x\cdot\dif u}^2\\
 =&\Div\left[(-t)G\hin{\partial_tu+\dfrac {1}{2t}x\cdot\dif u}{\dif u}+\dfrac x4\abs{\dif u}^2G\right].
\end{align*}
Choose $\eta_R\in C_0^\infty\left(B_{2R}\right)$ satisfies
\begin{align*}
    \eta_R\vert_{B_R}=1,\quad 0\leq\eta_R\leq 1,\quad \abs{\nabla\eta_R}\leq C/R.
\end{align*}
Then
\begin{align*}
    &\dfrac{\dif}{\dif t}\left(\dfrac{-t}{2}\int_{\mathbb{R}^m}\abs{\dif u}^2G\eta_R^2\right)
 +(-t)\int_{\mathbb{R}^m}\abs{u_t+\dfrac{1}{2t}x\cdot\dif u}^2G\eta_R^2\\
=&-2(-t)\int_{\mathbb{R}^m}\eta_RG
 \hin{u_t+\dfrac{1}{2t}x\cdot\dif u}{du(\nabla\eta_R)}-\dfrac12\int_{\mathbb{R}^m}\eta_R(x\cdot\nabla\eta_R)\abs{\dif u}^2G.
\end{align*}
Integrating in time and applying Young's inequality gives
\begin{align*}
    \dfrac12\int_{t_1}^{t_2}(-t)\int_{\mathbb{R}^m}
 \abs{u_t+\dfrac{1}{2t}x\cdot\dif u}^2G\eta_R^2\le\Phi(t_1)
 +C\int_{t_1}^{t_2}\int_{B_{2R}\setminus B_R}
 \left(\dfrac{-t}{R^2}+1\right)\abs{\dif u}^2G.
\end{align*}
Monotone convergence proves finiteness of the full dissipation integral.
The boundary errors now satisfy
\begin{align*}
    &\abs{2\int_{t_1}^{t_2}(-t)\int_{\mathbb{R}^m}\eta_RG
 \hin{u_t+\dfrac{1}{2t}x\cdot\dif u}{du(\nabla\eta_R)}}\\
 \le&\dfrac C R\left(\int_{t_1}^{t_2}(-t)\int_{\mathbb{R}^m}\abs{u_t+\dfrac{1}{2t}x\cdot\dif u}^2G\right)^{1/2}
 \left(\int_{t_1}^{t_2}\Phi(t)\dif t\right)^{1/2}\to 0,
\end{align*}
and
\begin{align*}
\abs{\dfrac12\int_{t_1}^{t_2}\int_{\mathbb{R}^m}\eta_R(x\cdot\nabla\eta_R)\abs{\dif u}^2G}\le C\int_{t_1}^{t_2}\int_{\mathbb{R}^m\setminus B_R}\abs{\dif u}^2G
\to 0.
\end{align*}
Letting $R\to\infty$ proves \eqref{eq:monotonicity} and local absolute
continuity.
\end{proof}

\begin{prop}\label{prop:moment}
If $\Phi$ is locally bounded on $(-\infty,0)$, then for almost every $t<0$,
\begin{equation}\label{eq:moment}
 \int\abs{x}^2\abs{\dif u}^2G_{\sqrt{-t}}(x)\dif x \le8(m-2)\Phi(t)
 +4\int\abs{x\cdot  du+2tu_t}^2G_{\sqrt{-t}}(x)\dif x.
\end{equation}
\end{prop}

\begin{proof}
Equation \eqref{eq:monotonicity} gives
$\int\abs{x\cdot\dif u+2tu_t}^2G_{\sqrt{-t}}(x)<\infty$ for almost every $t<0$.
Fix such a time and write $G=G_{\sqrt{-t}}$. For
$X\in C_0^1(\mathbb{R}^m,\mathbb{R}^m)$, the conservation law for the stress–energy tensor gives
\begin{align}\label{eq:stress}
\begin{split}
     &\int_{\mathbb{R}^m}\left[
 \left(\Div X+\dfrac{1}{2t}x\cdot X\right)\abs{\dif u}^2
 -2\sum_{i,k}\hin{\partial_i u}{\partial_k u}\partial_iX^k\right]G\dif x\\
 =&\dfrac1t\int_{\mathbb{R}^m}\hin{ x\cdot\dif u+2tu_t}{\dif u(X)} G\dif x.
\end{split}
\end{align}
Take $X=x\eta_R^2$, with the cutoffs $\eta_R$ above. Young's inequality and
$\abs{x}\abs{\nabla\eta_R}\le C$ give
\begin{align*}
 \dfrac1{4(-t)}\int_{\mathbb{R}^m}\abs{x}^2\abs{\dif u}^2\eta_R^2G
 \le&\dfrac{2(m-2)}{-t}\Phi(t)
 +\dfrac1{-t}\int_{\mathbb{R}^m}\abs{x\cdot\dif u+2tu_t}^2G\\
 &+C\int_{B_{2R}\setminus B_R}\abs{\dif u}^2G.
\end{align*}
The last term tends to zero as  $R\to\infty$. Monotone convergence proves \eqref{eq:moment}.
\end{proof}
For $a\in\mathbb{R}^m$, $T\in\mathbb{R}$, and $t<\min\set{T,0}$, write
\begin{align*}
    \Phi_{a,T}(t)=\dfrac{T-t}{2}\int_{\mathbb{R}^m} \abs{\dif u(x,t)}G_{\sqrt{T-t}}(x-a)\dif x.
\end{align*}
The change of variables
$\widehat u(x,s)=u(\lambda x+a,\lambda^2s+T)$ gives
\begin{align*}
    \partial_s\widehat u=\tau(\widehat u),\qquad \dfrac{-s}{2}\int_{\mathbb{R}^m}\abs{\dif \widehat u(x,s)}^2G_{\sqrt{-s}}(x)\dif x =\Phi_{a,T}(\lambda^2s+T).
\end{align*}
Applying the computation in Proposition~\ref{prop:monotonicity} on compact time intervals where $\widehat u$ is defined shows that $\Phi_{a,T}$ is nonincreasing whenever it is locally bounded.

\begin{prop}\label{prop:global}
Assume $\sup_{t<0}\Phi(t)<\infty$. Then the finite limit $L=\lim\limits_{t\to-\infty}\Phi(t)$ exists and for every fixed $a\in\mathbb R^m, T\in\mathbb R$,
\begin{equation}\label{eq:all-center-limit}
\lim_{t\to-\infty}\Phi_{a,T}(t)=L.
\end{equation}
\end{prop}

\begin{proof}
Monotonicity gives the finite limit $L$ and
\begin{align*}
    \int_1^\infty\dfrac 1{4h}\int
 \abs{x\cdot\dif u(x,-h)-2h\partial_tu(x,-h)}^2G_{\sqrt h}(x)\dif x\dif h =L-\Phi(-1)<\infty.
\end{align*}
Choose $h_i\to\infty$ in the full-measure set where
\eqref{eq:moment} holds and such that
\begin{align*}
     \int\abs{x\cdot\dif u(x,-h_i)-2h_i \partial_tu(x,-h_i)}^2G_{\sqrt{h_i}}(x)\dif x\to0.
\end{align*}
By Proposition~\ref{prop:moment},
\begin{align*}
    \sup_i\int \abs{x}^2\abs{\dif u(x,-h_i)}^2G_{\sqrt{h_i}}(x)\dif x<\infty.
\end{align*}

First suppose $T<0$.  For $h>-T$ put
\begin{align*}
    K_h(x)=\dfrac {(h+T)G_{\sqrt{h+T}}(x-a)}{hG_{\sqrt h}(x)}.
\end{align*}
A direct calculation gives
\begin{align*}
     K_h(x)=(1+T/h)^{1-m/2}
 \exp\set{\dfrac{T}{4h(h+T)} \abs{x+\dfrac hT a}^2-\dfrac{\abs{a}^2}{4T}}.
\end{align*}
For every compact interval of $h$ contained in $(-T,\infty)$, the negative quadratic term in the exponent shows that $K_h$ is uniformly
bounded; hence $\Phi_{a,T}$ is locally bounded and the centered
monotonicity formula applies.  In particular, for all sufficiently large
$h$, $K_h\le C(a,T)$.  Moreover, for each fixed $R>0$,
\begin{align*}
    \sup_{\abs{x}\le R\sqrt h}\abs{K_h(x)-1}\to0
\end{align*}
Using this bound and splitting at
$\abs{x}=R\sqrt{h_i}$,
\begin{align*}
 \abs{\Phi_{a,T}(-h_i)-\Phi(-h_i)}
 \le&L\sup_{\abs{x}\le R\sqrt{h_i}}\abs{K_{h_i}(x)-1}\\
 &\quad+\dfrac {C(a,T)}{R^2} \int \abs{x}^2\abs{\dif u(x,-h_i)}^2G_{\sqrt{h_i}}(x)\dif x.
\end{align*}
Letting first $i\to\infty$ and then $R\to\infty$ gives
\begin{align*}
     \Phi_{a,T}(-h_i)\to L.
\end{align*}

Fix now $S<0$.  For $t<S$, the Gaussian semigroup identity gives
\begin{align*}
G_{\sqrt{-t}}(x-a)=\int_{\mathbb R^m}G_{\sqrt{S-t}}(x-b)G_{\sqrt{-S}}(b-a)\dif b,
\end{align*}
and hence
\begin{align*}
 \Phi_{a,0}(t)
 =\dfrac {-t}{S-t}\int_{\mathbb R^m}\Phi_{b,S}(t)G_{\sqrt{-S}}(b-a)\dif b.
\end{align*}
For every fixed $b$, the first part gives $\Phi_{b,S}(t)\to L$ as $t\to-\infty$, and monotonicity gives $0\le\Phi_{b,S}(t)\le L$.  The same identity shows that $\Phi_{a,0}$ is locally bounded on $(-\infty,0)$, so its centered monotonicity formula is legitimate.  Dominated convergence yields
\begin{align*}
    \lim_{t\to-\infty}\Phi_{a,0}(t)=L.
\end{align*}

Finally let $T>0$.  Again by the Gaussian semigroup identity,
\begin{align*}
    \Phi_{a,T}(t)
 =\left(1+\dfrac {T}{-t}\right)
 \int_{\mathbb R^m}\Phi_{b,0}(t)G_{\sqrt T}(b-a)\dif b.
\end{align*}
The $T=0$ case and monotonicity give
$0\le\Phi_{b,0}(t)\le L$ and
$\Phi_{b,0}(t)\to L$ for every fixed $b$.  Thus the right-hand side is
finite and locally bounded in $t$; consequently $\Phi_{a,T}$ also satisfies the centered monotonicity formula.  Dominated convergence proves
\eqref{eq:all-center-limit} for $T>0$ as well.
\end{proof}

\begin{cor}\label{cor:local}
Under the hypotheses of Proposition~\ref{prop:global},
\begin{equation}\label{eq:local-energy}
 \int_{B_R(a)}\abs{\dif  u(x,t)}^2\dif x\le C_mLR^{m-2} \qquad(a\in\mathbb{R}^m,\ t<0,\ R>0),
\end{equation}
and
\begin{equation}\label{eq:tail}
 \dfrac {-t}{2}\int_{\abs{x}\ge R\sqrt{-t}}\abs{\dif  u(x,t)}^2G_{\sqrt{-t}}(x)\dif x \le2^{m/2-1}Le^{-R^2/8}.
\end{equation}
For $s_1<s_2\le0$,
\begin{equation}\label{eq:time-energy}
\int_{s_1}^{s_2}\int_{B_R(a)}\abs{\partial_tu}^2\dif x\dif t \le C_mL\set{R^{m-2}+(s_2-s_1)R^{m-4}}.
\end{equation}
At $s_2=0$ the integral is taken over the open time interval.
\end{cor}

\begin{proof}
Since $G_R(x-a)\ge(4\pi)^{-m/2}e^{-1/4}R^{-m}$ on $B_R(a)$,
\begin{align*}
     \dfrac {R^2}{2}(4\pi)^{-m/2}e^{-1/4}R^{-m} \int_{B_R(a)}\abs{\dif  u(x,t)}^2\dif x\le\Phi_{a,t+R^2}(t)\le L.
\end{align*}
On $\abs{x}\ge R\sqrt{-t}$,
\begin{align*}
   G_{\sqrt{-t}}(x)\le2^{m/2}e^{-R^2/8}G_{\sqrt{-2t}}(x),\\
 \dfrac{-t}{2}\int\abs{\dif  u(x,t)}^2G_{\sqrt{-2t}}(x)\dif x=\Phi_{0,-t}(t)\le L.
\end{align*}
These prove \eqref{eq:local-energy} and \eqref{eq:tail}. Take $\eta\in C_0^\infty(B_{2R}(a))$, with $0\le\eta\le1$, $\eta=1$ on $B_R(a)$ and $\abs{\dif\eta}\le C/R$. The local energy identity gives
\begin{align*}
\dfrac {\dif}{\dif t}\int\eta^2\abs{\dif u}^2+\int\eta^2\abs{\partial_tu}^2 \le4\int\abs{\dif u(\nabla\eta)}^2.
\end{align*}
For $s_1<s_2<0$, integration and \eqref{eq:local-energy} yield the estimate below; monotone convergence then permits $s_2\uparrow0$:
\begin{align*}
\int_{s_1}^{s_2}\int_{B_R(a)}\abs{\partial_tu}^2\le&\int_{B_{2R}(a)}\abs{\dif u(x,s_1)}^2\dif x +\dfrac C{R^2}\int_{s_1}^{s_2}\int_{B_{2R}(a)}\abs{\dif u}^2\\
 \le& C_mL\{R^{m-2}+(s_2-s_1)R^{m-4}\}.
\end{align*}
\end{proof}

\begin{prop}\label{prop:gaussian-local-equivalence}
For an ancient harmonic map heat flow, the following two conditions are equivalent:
\begin{equation}\label{eq:gaussian-condition}
 \sup_{t<0}\Phi(t)<\infty,
\end{equation}
and
\begin{equation}\label{eq:local-spacetime-condition}
 \sup_{a\in\mathbb R^m,\ t_0<0,\ r>0}
 \set{r^{-m}\int_{t_0-4r^2}^{t_0} \int_{B_{2r}(a)}\abs{\dif u}^2 +r^{2-m}\int_{t_0-4r^2}^{t_0} \int_{B_{2r}(a)}\abs{\partial_tu}^2}<\infty.
\end{equation}
More precisely, the two quantities in \eqref{eq:gaussian-condition} and
\eqref{eq:local-spacetime-condition} bound each other up to a constant depending
only on $m$.
\end{prop}

\begin{proof}
Assume first \eqref{eq:gaussian-condition} and put
$L=\lim_{t\to-\infty}\Phi(t)$.  By monotonicity \eqref{eq:monotonicity},
$L=\sup_{t<0}\Phi(t)$. Proposition~\ref{prop:global} and Corollary~\ref{cor:local} give, uniformly in $a,t_0,r$,
\begin{align*}
  r^{-m}\int_{t_0-4r^2}^{t_0} \int_{B_{2r}(a)}\abs{\dif u}^2\le C_mL,
\end{align*}
and
\begin{align*}
r^{2-m}\int_{t_0-4r^2}^{t_0}\int_{B_{2r}(a)}\abs{\partial_tu}^2\le C_mL.
\end{align*}
Thus \eqref{eq:local-spacetime-condition} holds.

Conversely, let $A$ denote the supremum in
\eqref{eq:local-spacetime-condition}.  For fixed $a\in\mathbb R^m$,
$t_0<0$ and $r>0$, choose $s\in(t_0-r^2,t_0)$ so that
\begin{align*}
  \int_{B_{2r}(a)}\abs{\dif u(x,s)}^2\dif x \le r^{-2}\int_{t_0-r^2}^{t_0} \int_{B_{2r}(a)}\abs{\dif u}^2\le Ar^{m-2}.
\end{align*}
Let $\eta\in C_0^\infty(B_{2r}(a)), 0\leq\eta\leq 1$, $\eta=1$ on $B_r(a)$ and $\abs{\dif\eta}\le C/r$.  The local energy inequality gives
\begin{align*}
   \dfrac {d}{dt}\int\eta^2\abs{\dif u}^2
 \le \dfrac {C}{r^2}\int_{B_{2r}(a)}\abs{\dif u}^2.
\end{align*}
After integration from $s$ to $t_0$,
\begin{equation}\label{eq:morrey-from-spacetime}
 \int_{B_r(a)}\abs{\dif u(x,t_0)}^2\dif x\le C_mAr^{m-2}.
\end{equation}
Fix $t<0$ and set $R=\sqrt{-t}$.  On
$B_{2^{j+1}R}\setminus B_{2^jR}$,
\begin{align*}
G_R(x)\le (4\pi R^2)^{-m/2}e^{-4^j/4}.
\end{align*}
Using \eqref{eq:morrey-from-spacetime} on $B_R$ and on the balls
$B_{2^{j+1}R}$ gives
\begin{align*}
\Phi(t)\le C_mA\left(1+\sum_{j=0}^\infty 2^{j(m-2)}e^{-4^j/4}\right)\le C_mA.
\end{align*}
Taking the supremum over $t<0$ proves \eqref{eq:gaussian-condition} and the quantitative equivalence.
\end{proof}

\section{A compact-lift quasi-harmonic sphere}

Fix a compact set $K\subset N$. Without loss of generality, assume $N$ is closed and  fix an isometric embedding $N\hookrightarrow\mathbb R^\nu$.

\begin{lem}\label{lem:compact-lift}
Let $m\ge2$, and let $u:\mathbb R^m\times(-\infty,0)\to K$ be a nonconstant ancient harmonic map heat flow. Assume
\begin{align*}
\sup_{t<0}\Phi(t)<\infty,
\end{align*}
and that $K$ contains no harmonic $k$-sphere for all $2\leq k\leq \max\set{2,m-1}$.
Then $m\ge3$, and for some $3\le\ell\le m$, $K$ contains a
quasi-harmonic $\ell$-sphere
\begin{align*}
v:\mathbb R^\ell\to K.
\end{align*}
Moreover, for every Riemannian covering $\pi:\widetilde N\to N$, every lift $\widetilde v:\mathbb R^\ell\to\widetilde N$ of $v$ has relatively compact image.
\end{lem}

\begin{proof}
\noindent\textbf{Step 1.  Extraction of a least-dimensional quasi-harmonic sphere.}

Set
\begin{align*}
L=\lim_{t\to-\infty}\Phi(t).
\end{align*}
Then $0<L<\infty$.  Indeed, monotonicity \eqref{eq:monotonicity} gives $0\le\Phi(t)\le L$ for every $t<0$, while $L=0$ would imply $\dif u\equiv0$ and hence $\partial_tu=\tau(u)=0$.

Choose $R_i\to\infty$ and put
\begin{align*}
 u_i(x,t)=u(R_i x,R_i^2t).
\end{align*}
For every $0<a<b$,
\begin{align*}
\Phi(R_i^2t)\to L, \qquad \Phi(-bR_i^2)-\Phi(-aR_i^2)\to0.
\end{align*}
For $x_0\in\mathbb R^m$, $t_0<0$ and $r>0$, consider the recalling
\begin{align*}
\widehat  u_i(y,s)=u_i(x_0+ry,t_0+r^2s), \qquad (y,s)\in B_2\times(-4,0).
\end{align*}
Proposition~\ref{prop:gaussian-local-equivalence} gives the scale-invariant
bound
\begin{equation}\label{eq:uniform-local-H1}
 \int_{-4}^0 \int_{B_2}
 \left(\abs{\dif\widehat u_i}^2+\abs{\partial_s\widehat u_i}^2\right) \le C_mL,
\end{equation}
uniformly in $i,x_0,t_0,r$.  Hence, after a diagonal subsequence,
\begin{align*}
    u_i\rightharpoonup u_\infty\quad\hbox{in }H^1_{\rm loc}, \qquad u_i\to u_\infty\quad\hbox{in }L^2_{\rm loc}.
\end{align*}
The limit is $K$-valued almost  everywhere.  We now invoke the compactness argument in \cite[Proposition~4.1, p.~424]{Lin:1999ab}.  If the convergence were not strong in $H^1_{\rm loc}$, then a step-by-step examination of the proof of Claims~1--7 there would produce a nonconstant ancient two-dimensional flow
\begin{align*}
    q:\mathbb R^2\times(-\infty,0)\to K
\end{align*}
satisfying
\begin{equation}\label{eq:LW-terminal-bounds}
 \sup_{s<0}\int_{\mathbb R^2}\abs{\dif q(\cdot,s)}^2<\infty, \qquad  \sup_{\mathbb R^2\times(-\infty,0)}\abs{\dif q}<\infty, \qquad \int_{-\infty}^0 \int_{\mathbb R^2}\abs{\partial_sq}^2<\infty.
\end{equation}
These are exactly the hypotheses of Claim~8 of Lin--Wang.  That claim produces times $s_j\to-\infty$ and a finite-energy harmonic limit $q_{-\infty}:\mathbb R^2\to K$ and concludes, using the energy inequality,
that $q$ is constant whenever $q_{-\infty}$ is constant \cite[pp.~422--423]{Lin:1999ab}.  The removable singularity theorem extends $q_{-\infty}$ across infinity to a harmonic map $\mathbb S^2\to N$ with image in $K$; by hypothesis it is constant.  Hence Claim~8 contradicts the
choice of $q$.  Consequently
\begin{equation}\label{eq:strong-blowdown}
 u_i\to u_\infty \quad\hbox{strongly in }H^1_{\rm loc}
 (\mathbb R^m\times(-\infty,0)).
\end{equation}

By the monotonicity formula \eqref{eq:monotonicity},
\begin{align*}
    \int_{-b}^{-a}(-t)\int
 \abs{\partial_tu_i+\dfrac {1}{2t}x\cdot\dif u_i}^2G_{\sqrt{-t}}
 \to0.
\end{align*}
Passing to the limit in view of \eqref{eq:strong-blowdown}, \begin{equation}\label{eq:selfsimilar-blowdown}
 x\cdot\dif u_\infty+2t\partial_tu_\infty=0,
 \qquad
 u_\infty(x,t)=V \left(\dfrac {x}{\sqrt{-t}}\right).
\end{equation}
Moreover,
\begin{align*}
     \int_{-b}^{-a}\Phi(R_i^2t)\,\dif t\to(b-a)L.
\end{align*}
Strong convergence on fixed balls and the uniform Gaussian tail estimate \eqref{eq:tail} therefore imply
\begin{equation}\label{eq:profile-energy}
 \dfrac 12\int_{\mathbb R^m}\abs{\dif V}^2G_1\,\dif x=L>0.
\end{equation}
Thus $V$, and hence $u_\infty$, is  nonconstant.
If $m=2$, Proposition~\ref{prop:moment} and the monotonicity formula give
\begin{align*}
 \int_{-b}^{-a} \int \abs{x}^2\abs{\dif u_i}^2G_{\sqrt{-t}}\,\dif x\dif t \le16b\,[\Phi(-bR_i^2)-\Phi(-aR_i^2)]\to0.
 \end{align*}
Together with \eqref{eq:strong-blowdown} on every annulus this yields
$\dif V=0$, contradicting \eqref{eq:profile-energy}.  Hence $m\ge3$.

We next use the local dimension reduction of Lin--Wang.  All its analytic hypotheses are already available: \eqref{eq:uniform-local-H1}, the three properties (3.11), (3.11'') and (3.12) of \cite{Lin:1999ab}, and the strong local
compactness just proved.  Thus the  argument of \cite[Theorem~4.3, pp.~425--427]{Lin:1999ab}, applied to a singular self-similar limit, yields one of the two alternatives
\begin{equation}\label{eq:LW-alternative}
 \begin{cases}
 \omega:\mathbb S^k\to K \text{ is a nonconstant harmonic map},
     &2\le k\le m-1,\\[2mm]
 v:\mathbb R^k\to K \text{ is a nonconstant quasi-harmonic sphere},
     &3\le k\le m.
 \end{cases}
\end{equation}
The first alternative in \eqref{eq:LW-alternative} is excluded
by hypothesis. If $u_\infty$ is smooth, its profile $V$ in \eqref{eq:selfsimilar-blowdown} is already the second alternative with
$k=m$. We may therefore choose a nonconstant quasi-harmonic sphere
of least possible dimension and denote it by
\begin{equation}\label{eq:minimal-qhs}
 v:\mathbb R^\ell\to K,
 \qquad 3\le\ell\le m.
\end{equation}
Every map in the reduction is a strong local limit of parabolic rescalings
of $K$-valued maps; hence its image is contained in $K$. Choose such a quasi-harmonic sphere with $\ell$ minimal.

\vspace{2ex}

\noindent\textbf{Step 2. Vanishing of the terminal densities.}

Let
\begin{align*}
w(x,t)=v \left(\dfrac {x}{\sqrt{-t}}\right),\qquad t<0.
\end{align*}
For this flow, $\Phi_{a,T}$ has the meaning fixed in Section~2.  When the
Euclidean dimension changes below, $G_r$ always denotes the Gaussian kernel in the current dimension.  We claim that
\begin{equation}\label{eq:terminal-zero}
 \lim_{r\downarrow0}\Phi_{\xi,0}(-r^2)=0 \qquad (\xi\in\mathbb S^{\ell-1}).
\end{equation}
Assume otherwise.  For some $\xi\in\mathbb S^{\ell-1}$,
\begin{align*}
\Lambda:=\lim_{r\downarrow0}\Phi_{\xi,0}(-r^2)>0.
\end{align*}
Choose $r_i\downarrow0$ and set
\begin{align*}
    w_i(x,t)=w(\xi+r_i x,r_i^2t).
\end{align*}
For every $0<a<b$,
\begin{align}
 \dfrac{-t}{2}\int \abs{\dif w_i(x,t)}^2G_{\sqrt{-t}}(x)\,\dif x
 &=\Phi_{\xi,0}(r_i^2t)\to\Lambda,
 \label{eq:terminal-energy}\\
 \int_{-b}^{-a}(-t)\int
 \abs{\partial_tw_i+\dfrac {1}{2t}x\cdot\dif w_i}^2G_{\sqrt{-t}}
 &=\Phi_{\xi,0}(-br_i^2)-\Phi_{\xi,0}(-ar_i^2)
 \to0.
 \label{eq:terminal-defect}
\end{align}
The same compactness argument as above gives, after a subsequence,
\begin{align*}
w_i\to\bar w \quad\hbox{strongly in }H^1_{\rm loc}.
\end{align*}
From \eqref{eq:terminal-defect},
\begin{align*}
x\cdot \dif \bar w+2t\bar w_t=0, \qquad \bar w(x,t)=\bar v \left(\dfrac {x}{\sqrt{-t}}\right).
\end{align*}
Passing to the limit in \eqref{eq:terminal-energy}, first on fixed balls
and then using the Gaussian tail estimate, gives
\begin{equation}\label{eq:terminal-profile-energy}
 \dfrac 12\int_{\mathbb R^\ell}\abs{\dif\bar v}^2G_1\,\dif x=\Lambda>0.
\end{equation}
Since the original flow $w$ is self-similar,
\begin{align*}
     x\cdot\dif w+2tw_t=0.
\end{align*}
Evaluating this identity at $(\xi+r_i x,r_i^2t)$ yields the exact relation
\begin{equation}\label{eq:direction-loss}
 \partial_\xi w_i =-r_i\left(x\cdot\dif w_i+2t\partial_tw_i\right).
\end{equation}
Hence for every compact cylinder
$Q\Subset\mathbb R^\ell\times(-\infty,0)$,
\begin{align*}
   \int_Q\abs{\partial_\xi w_i}^2
 \le C_Qr_i^2\int_Q\left(\abs{\dif w_i}^2+\abs{\partial_tw_i}^2\right) \to0.
\end{align*}
Thus $\partial_\xi\bar w=0$. Identifying $\xi^\perp$ with $\mathbb R^{\ell-1}$ and using the same symbols for the induced flow and its profile, Fubini's theorem gives
\begin{equation}\label{eq:dimension-drop-energy}
 \dfrac 12\int_{\mathbb R^{\ell-1}}\abs{\dif\bar v}^2G_1\,\dif y
 =\Lambda>0.
\end{equation}

If $\ell\ge4$ and $\bar w$ is smooth, then $\bar v$ itself is a quasi-harmonic sphere of dimension $\ell-1$, contradicting the minimality in \eqref{eq:minimal-qhs}.  If $\bar w$ is singular, the same localized Lin--Wang reduction gives
\begin{align*}
 \begin{cases}
 \mathbb S^k\to K\text{ nonconstant harmonic},&2\le k\le\ell-2,\\
 \mathbb R^k\to K\text{ nonconstant quasi-harmonic},&3\le k\le\ell-1,
 \end{cases}
\end{align*}
again impossible.  It remains to consider $\ell=3$.  Then $\bar w$ is a
weak two-dimensional flow.  Under the exclusion of a $K$-valued harmonic
$\mathbb S^2$, the two-dimensional regularity argument in \cite[Proposition~4.2, p.~424] {Lin:1999ab} gives that $\bar w$ is smooth.  Applying Proposition~\ref{prop:moment} in dimension two and using exact self-similarity,
\begin{align*}
     y\cdot \dif \bar w+2t\bar w_t=0,
\end{align*}
we obtain, for almost every $t<0$,
\begin{align*}
\int_{\mathbb R^2}\abs{y}^2\abs{\dif \bar w(y,t)}^2G_{\sqrt{-t}}(y)\,\dif y
 \le4\int_{\mathbb R^2}\abs{y\cdot \dif \bar w+2t\bar w_t}^2
 G_{\sqrt{-t}}\,\dif y=0.
\end{align*}
Hence $\dif \bar w=0$, contradicting \eqref{eq:dimension-drop-energy}.
This proves \eqref{eq:terminal-zero}.

\vspace{2ex}

\noindent\textbf{Step 3. Decay at infinity and radial energy.}

For fixed $r>0$ and $\abs{\xi}=1$,
\begin{align*}
    G_r(x-\xi)\le2^{\ell/2}e^{1/(4r^2)}G_{\sqrt2r}(x).
\end{align*}
Proposition~\ref{prop:global} therefore gives an integrable majorant, so $\xi\mapsto\Phi_{\xi,0}(-r^2)$ is continuous.  Choose any sequence $r_j\downarrow0$.  By \eqref{eq:terminal-zero} and monotonicity,
\begin{align*}
    \Phi_{\xi,0}(-r_j^2)\downarrow0
 \qquad\hbox{for every }\xi\in\mathbb S^{\ell-1}.
\end{align*}
Dini's theorem gives
\begin{align*}
     \sup_{\xi\in\mathbb S^{\ell-1}}\Phi_{\xi,0}(-r_j^2)\to0.
\end{align*}
Since $r\mapsto\Phi_{\xi,0}(-r^2)$ is monotone, the limit is independent of the chosen sequence, and hence
\begin{equation}\label{eq:terminal-uniform}
 \sup_{\xi\in\mathbb S^{\ell-1}}\Phi_{\xi,0}(-r^2)
 \to0\qquad(r\downarrow0).
\end{equation}

Let $\varepsilon_0$ and $\delta\le1$ be the constants in \cite[Lemma~3.3, p.~413]{Lin:1999ab}.  Choose $r>0$ later.  For $\theta\in\mathbb S^{\ell-1}$ and $-\delta^2r^2<t_0<0$, put
\begin{equation}\label{eq:regularity-center}
 \tau=\dfrac {t_0}{2}-\delta^2r^2.
\end{equation}
Then
\begin{align*}
0<t_0-\tau<\delta^2r^2,
 \qquad
 \tau+\delta^2r^2=\dfrac {t_0}{2}<0,
\end{align*}
so $(\theta,t_0)\in P_{\delta r}(\theta,\tau)$ and the whole cylinder is contained in negative time.  For $r\le\lambda\le2r$, set $s=\tau-\lambda^2$. Since $-\tau\le\frac 32\delta^2r^2$,
\begin{align*}
 \dfrac {\lambda^2G_\lambda(x-\theta)}
 {(-s)G_{\sqrt{-s}}(x-\theta)}=
 \left(\dfrac {\lambda^2-\tau}{\lambda^2}\right)^{(\ell-2)/2}
 \exp \left[-\dfrac {\abs{x-\theta}^2}{4} \left(\dfrac 1{\lambda^2}-\dfrac 1{\lambda^2-\tau}\right)\right]\le3^{(\ell-2)/2}.
\end{align*}
Moreover $s\ge-6r^2$.  Hence, by monotonicity,
\begin{equation}\label{eq:center-comparison}
 \Phi_{\theta,\tau}(\tau-\lambda^2)
 \le3^{(\ell-2)/2}\Phi_{\theta,0}(-6r^2), \qquad r\le\lambda\le2r.
\end{equation}
Let $0\le\eta\le1$ be the fixed spatial cutoff in Lin--Wang's Lemma~3.3, chosen identically for $\theta\in\mathbb S^{\ell-1}$. From
\eqref{eq:center-comparison},
\begin{align}\label{eq:LW-smallness}
\begin{split}
    &\int_{\tau-4r^2}^{\tau-r^2} \int
 \eta^2\abs{\dif w}G_{\sqrt{\tau-s}}(x-\theta)\,\dif x\dif s\\
 \le&4\int_r^{2r}
 \Phi_{\theta,\tau}(\tau-\lambda^2)\dfrac {\dif\lambda}{\lambda}\\
 \le&4\cdot3^{(\ell-2)/2}\log2\,
 \sup_{\zeta\in\mathbb S^{\ell-1}}\Phi_{\zeta,0}(-6r^2).
\end{split}
\end{align}
By \eqref{eq:terminal-uniform}, fix $r>0$ so small that the last member of
\eqref{eq:LW-smallness} is less than $\varepsilon_0^2$.  Lemma~3.3 then
gives
\begin{align*}
\abs{\dif w(\theta,t_0)}\le\dfrac Cr \qquad (\theta\in\mathbb S^{\ell-1},\ -\delta^2r^2<t_0<0).
\end{align*}
Since
\begin{align*}
\dif w(\theta,-R^{-2})=R\,\dif v(R\theta),
\end{align*}
we obtain, after increasing the constant once and for all,
\begin{equation}\label{eq:decay}
 \abs{\dif v(R\theta)}\le\dfrac  CR, \qquad R\ge R_0,
\end{equation}
for some fixed $R_0$.

Corollary~\ref{cor:local} gives
$\int_{-1}^0\int_{B_1}\abs{w_t}^2<\infty$.  With $h=-t$ and $y=x/\sqrt h$,
\begin{align*}
w_t(x,-h)=\dfrac 1{2h}\,y\cdot \dif v(y).
\end{align*}
Writing $y=r\theta$ and applying Tonelli's theorem,
\begin{align} \label{eq:radial-identity}
\begin{split}
 &\int_{-1}^0 \int_{B_1}\abs{w_t}^2\,\dif x\dif t\\
 =&\dfrac 14\int_0^1h^{\ell/2-2}
 \int_0^{h^{-1/2}}r^{\ell+1}
 \int_{\mathbb S^{\ell-1}}\abs{v_r(r\theta)}^2\,\dif\theta\dif r\dif h\\
 =&\dfrac 1{2(\ell-2)}\left[ \int_0^1r^{\ell+1}\int_{\mathbb S^{\ell-1}} \abs{v_r(r\theta)}^2\,\dif\theta\dif r
 +\int_1^\infty r^3\int_{\mathbb S^{\ell-1}} \abs{v_r(r\theta)}^2\,\dif\theta\dif r\right].
 \end{split}
\end{align}
Thus
\begin{equation}\label{eq:radial-integral}
 \int_1^\infty r^3\int_{\mathbb S^{\ell-1}} \abs{v_r(r\theta)}^2\,\dif\theta\dif r<\infty.
\end{equation}

\vspace{2ex}
\noindent\textbf{Step 4. Compactness of the lifted image.}

Fix a Riemannian covering $\pi:\widetilde N\to N$ and a lift
$\widetilde v$.  Since $\pi$ is a local isometry, \eqref{eq:decay} gives
\begin{equation}\label{eq:lift-derivatives}
 \abs{\partial_R\widetilde v(R\theta)}=\abs{v_R(R\theta)}, \qquad \abs{d_\theta\widetilde v(R\theta)}\le C \quad(R\ge R_0).
\end{equation}
For fixed $s\ge R\ge R_0$, let
\begin{align*}
F(\theta)=d_{\widetilde N} \left(\widetilde v(s\theta),\widetilde v(R\theta)\right).
\end{align*}
Then $\operatorname{Lip}F\le2C$, and
\begin{align}\label{eq:lift-L2}
\begin{split}
\int_{\mathbb S^{\ell-1}}F^2\,\dif\theta\le&\int_{\mathbb S^{\ell-1}}
\left(\int_R^s\abs{v_r(r\theta)}\,\dif r\right)^2\dif\theta \\
 \le&\left(\int_R^sr^{-3}\,\dif r\right) \left(\int_R^sr^3\int_{\mathbb S^{\ell-1}} \abs{v_r(r\theta)}^2\,\dif\theta\dif r\right) \\
 \le&\dfrac 1{2R^2}\int_R^\infty r^3
 \int_{\mathbb S^{\ell-1}}\abs{v_r(r\theta)}^2\,\dif\theta\dif r \to0
\end{split}
\end{align}
uniformly for $s\ge R$. If $\norm{F}_{L^\infty}>0$ and
$F(\theta_0)=\norm{F}_{L^\infty}$, then $F\ge \norm{F}_{L^\infty}/2$ on the geodesic ball
\begin{align*}
B_{\min\set{1,\norm{F}_{L^\infty}/(4C)}}^{\mathbb S^{\ell-1}}(\theta_0).
\end{align*}
Therefore, for a dimensional constant $c_\ell>0$,
\begin{equation}\label{eq:L2-Linf}
 \norm{F}_{L^2}^2 \ge c_\ell M^2\min\set{1,\left(\dfrac {\norm{F}_{L^\infty}}{4C}\right)^{\ell-1}}.
\end{equation}
Equations \eqref{eq:lift-L2}--\eqref{eq:L2-Linf} imply
\begin{equation}\label{eq:lift-tail}
 \sup_{s\ge R}\sup_{\theta\in\mathbb S^{\ell-1}}
 d_{\widetilde N}\left(\widetilde v(s\theta),\widetilde v(R\theta)\right)
 \to0.
\end{equation}

It remains to prove that the connected component of $\pi^{-1}(K)$ containing
$\widetilde v(\mathbb R^\ell)$ is complete.  Let $(z_j)$ be a Cauchy sequence in that component.  Since $\pi$ is distance nonincreasing,
\begin{align*}
     d_N(\pi z_j,\pi z_k)\le d_{\widetilde N}(z_j,z_k),
\end{align*}
so, after passage to the limit in the compact set $K$, $\pi z_j\to z\in K$.  Choose an evenly covered neighborhood $U$ of $z$ and $r_*>0$ such that $B_N(z,4r_*)\subset U$.  For some fixed large $j_0$ and all large $j$,
\begin{align*}
  \pi z_j\in B_N(z,r_*), \qquad d_{\widetilde N}(z_j,z_{j_0})<r_*.
\end{align*}
A curve joining $z_{j_0}$ to $z_j$ with length $<2r_*$ projects into
$B_N(z,3r_*)\subset U$.  Hence all sufficiently large $z_j$ lie in the same sheet $\widetilde U$ of $\pi^{-1}(U)$, and
\begin{align*}
z_j=(\pi|_{\widetilde U})^{-1}(\pi z_j)
 \to(\pi|_{\widetilde U})^{-1}(z).
\end{align*}
Thus the component is complete.  Finally, \eqref{eq:lift-tail} implies that for every $\varepsilon>0$ and all sufficiently large $R$,
\begin{align*}
 \widetilde v(\mathbb R^\ell\setminus B_R)
 \subset\mathcal N_{\varepsilon/2}
 \left(\widetilde v(\partial B_R)\right).
\end{align*}
The compact set $\widetilde v(\overline B_R)$ has a finite
$\varepsilon/2$-net, which is then an $\varepsilon$-net for the whole image.
Hence the image is totally bounded in a complete metric space and therefore
has compact closure.
\end{proof}

Now we can state the rigidity result for ancient harmonic map heat flows. 
\begin{proof}[Proof of Theorem~\ref{thm:ancient}]
If $\omega:\mathbb S^k\to N$, $k\ge2$, is harmonic with image in $K$, then, since $k\ge2$, it lifts to
$\widetilde\omega:\mathbb S^k\to\widetilde N$.  By the
convex-supporting hypothesis there is a $C^2$ function $\rho$ on a
neighborhood of $\widetilde\omega(\mathbb S^k)$ with positive-definite
Hessian.  Hence, for some $c>0$,
\begin{align*}
   0=\int_{\mathbb S^k}\Delta(\rho\circ\widetilde\omega)
 =\int_{\mathbb S^k}\widetilde\nabla^2\rho
 (d\widetilde\omega,d\widetilde\omega)
 \ge c\int_{\mathbb S^k}\abs{\dif\widetilde\omega}^2,
\end{align*}
so every such harmonic sphere is constant.

If $u$ were nonconstant, then Lemma~\ref{lem:compact-lift} would give a nonconstant quasi-harmonic sphere $v:\mathbb R^\ell\to K$, $3\le\ell\le m$, whose lift $\widetilde v$ has compact image closure.  Choose a strictly convex
function $\rho$ on a neighborhood of that closure and put $f=\rho\circ\widetilde v$.  Since $\pi$ is a local isometry,
\begin{align*}
  \left(\Delta-\dfrac 12x\cdot\nabla\right)f  =\widetilde\nabla^2\rho(d\widetilde v,d\widetilde v) \ge c\abs{\dif v}^2
\end{align*}
for some $c>0$.  Since the lifted image has compact closure, $f$ is bounded.  Choose $\eta_R=1$ on $B_R$, supported in $B_{2R}$, with $\abs{\nabla\eta_R}\le C/R$ and $\abs{\Delta\eta_R}\le C/R^2$.
Writing $\mathcal L=\Delta-\dfrac 12x\cdot\nabla$, Gaussian integration by parts gives
\begin{align*}
   \int \eta_R\mathcal L f\,G_1
 =\int f\mathcal L\eta_R\,G_1.
\end{align*}
On $B_{2R}\setminus B_R$ we have
$\abs{\mathcal L\eta_R}\le C$, hence
\begin{align*}
    0\le c\int_{B_R}\abs{\dif v}^2G_1
 \le\abs{\int f\mathcal L\eta_R\,G_1} \le C\norm{f}_{\infty} \int_{B_{2R}\setminus B_R}G_1\,\dif x
 \to0.
\end{align*}
Thus $\dif v=0$, a contradiction.  Hence $u$ is constant.
\end{proof}

Finally, we give a proof of the nonexistence result for quasi-harmonic spheres. 
\begin{proof}[Proof of Theorem ~\ref{thm:shrinker}]
Let $K=\overline{v(\mathbb R^m)}$. For
\begin{align*}
    u(x,t)=v \left(\dfrac {x}{\sqrt{-t}}\right),\qquad
 y=\dfrac {x}{\sqrt{-t}},
\end{align*}
we have
\begin{align*}
    \partial_tu=\dfrac 1{2(-t)}y\cdot \dif v(y)
 =\dfrac 1{-t}\tau(v)(y)=\tau(u),
\end{align*}
and
\begin{align*}
    (-t)\int_{\mathbb R^m}\abs{\dif u(x,t)}^2G_{\sqrt{-t}}(x)\,\dif x=\dfrac 12\int_{\mathbb R^m}\abs{\dif v}^2G_1\,\dif x<\infty.
\end{align*}
Theorem~\ref{thm:ancient} applies.
\end{proof}

\noindent\textbf{Acknowledgments.}

J. Li is supported by NSFC No. 12531002, 12431004, 11721101.
L. Sun acknowledges support from the National Natural Science Foundation of China (Grant Nos.~12571055 and 12671062), the Natural Science Foundation of Hunan Province (Grant No.~2026JJ20014), the 111 Project (Grant No.~D23017), and the Program for
Science and Technology Innovative Research Team in Higher Educational Institutions of Hunan Province, China.

\vspace{2ex}

\noindent\textbf{Declaration on AI assistance.}
The authors used AI models, principally ChatGPT 5.6, to assist in the development of Step~2 in the proof of Lemma~3.1 which was independently verified and completed by the authors, who take full responsibility for the content.

%\bibliographystyle{amsplain} %
%\bibliography{LS}

\providecommand{\bysame}{\leavevmode\hbox to3em{\hrulefill}\thinspace}
\providecommand{\MR}{\relax\ifhmode\unskip\space\fi MR }
\providecommand{\MRhref}[2]{%
  \href{http://www.ams.org/mathscinet-getitem?mr=#1}{#2}
}
\providecommand{\href}[2]{#2}

\end{document}